\documentclass[11pt]{amsart}
\usepackage{amssymb}
\usepackage{amsmath}
\usepackage[active]{srcltx}
\usepackage{t1enc}
\usepackage[latin2]{inputenc}
\usepackage{verbatim}
\usepackage{amsmath,amsfonts,amssymb,amsthm}
\usepackage[mathcal]{eucal}
\usepackage{enumerate}
\usepackage[centertags]{amsmath}
\usepackage{graphicx}
\usepackage{color}
\graphicspath{ {./images/} }

\newtheorem{theorem}{Theorem}

\newtheorem{corollary}{Corollary}

\begin{document}
\author{N.Areshidze and G. Tephnadze}

\title[Approximation by Nörlund means in Lebesgue spaces ]{Approximation by Nörlund means with respect to Walsh system in Lebesgue spaces}
\address{N.Areshidze, Tbilisi State University, Faculty of Exact and Natural Sciences, Department of Mathematics, Chavchavadze str. 1, tbilisi 0128, georgia}
\email{nika.areshidze15@gmail.com}
\address {G. Tephnadze, The University of Georgia, School of Science and Technology, 77a Merab Kostava St, Tbilisi, 0128, Georgia.}
\email{g.tephnadze@ug.edu.ge}

\thanks{The research was supported by Shota Rustaveli National Science Foundation grant  FR-21-2844.}

\begin{abstract}
In this paper we improve and complement a result by M\'oricz and Siddiqi \cite{Mor}. In particular, we prove that their inequality of the Nörlund means with respect to the Walsh system holds also without their additional condition. Moreover, we prove some new approximation results and inequalities in Lebesgue spaces for any $1\leq p<\infty$.
\end{abstract}
\maketitle 
\date{}

\noindent \textbf{2010 Mathematics Subject Classification:} 42C10, 42B30.

\noindent \textbf{Key words and phrases:} Walsh group, Walsh system,  Fejér means, Nörlund means, approximation, inequalities.

\section{Introduction}

Concerning some definitions and notations used in this introduction we refer to Section 2.   
Fejér's theorem shows that (see e.g.  \cite{gat} and \cite{gol}) if one replaces ordinary summation by Fejér means $\sigma_n,$ defined by
\begin{equation*}
\sigma_n f:=\frac{1}{n}\sum_{k=1}^nS_kf, 
\end{equation*}
 then, for any $1\leq p\leq \infty,$ there exists an absolute constant $C_p,$ depending only on $p$ such that the inequality
\begin{equation*}
\left\Vert \sigma_nf\right\Vert_p\leq C_p\left\|f\right\|_{p}
\end{equation*}
holds.
Moreover, (see e.g. \cite{PTWbook})	let $1\leq p\leq \infty $, $2^N\leq n< 2^{N+1}$, $f\in L^p(G)$ and $n\in \mathbb{N}.$ Then the following inequality holds:
\begin{eqnarray}\label{aaa}
\left\Vert \sigma_{n}f-f\right\Vert_{p}
\leq 3\sum_{s=0}^{N} \frac{2^{s}}{2^{N}}\omega_p\left(1/2^s,f\right).
\end{eqnarray}

It follows that if $f\in lip\left( \alpha ,p\right) ,$ i.e.
	\begin{equation*}
	\omega_{p}\left( 1/2^{n},f\right) =O\left( 1/2^{n\alpha }\right), \ \text{as} \ n\rightarrow \infty,
	\end{equation*}
	then
	
	\begin{equation*}
	\left\Vert\sigma_nf-f\right\Vert _{p}=\left\{
	\begin{array}{c}O\left(1/2^N\right), \text{ \ \ \ if \ \ \ }\alpha >1,\\
	O\left(N/2^N\right),\text{ \ \ \ if \ \ \ }\alpha=1,\\
	O\left( 1/2^{N\alpha }\right) ,\text{ \ \ \ \ \ \ \ if }\alpha <1.\end{array}\right.
	\end{equation*}
Moreover, (see \cite{PTWbook}) if $1\leq p< \infty ,$ $f\in L^{p}(G)$ and
	\begin{equation*}
	\left\Vert \sigma _{2^{n}}f-f\right\Vert_{p}=o\left( 1/2^{n}\right), \ \text{as} \ n\rightarrow \infty,
	\end{equation*}
then $f$  is a constant function.

Boundedness of maximal operators of Vilenkin-Fejer means and  weak-$(1,1)$ type inequality 

\begin{equation*}
\mu \left( \sigma ^{*}f>\lambda \right) \leq \frac{c}{\lambda }\left\|
f\right\| _{1},\text{ \qquad }\left( f\in L_1(G), \ \ \lambda >0\right) 
\end{equation*}
can be found in Zygmund \cite{Zy}  for trigonometric series, in Schipp
\cite{Sc} for Walsh series and in Pál, Simon \cite{PS} and Weisz \cite{We1,We3}  for bounded Vilenkin
series. 

Convergence and summability of  N\"orlund means were studied by several authors. We mentioned Areshidze, Baramidze, Persson and Tephnadze \cite{ABPT}, Fridli, Manchanda,  Siddiqi \cite{FMS}, Persson, Tephnadze and Weisz \cite{PTWbook} (see also \cite{PSTW}), Blahota and Nagy \cite{BN} (see also \cite{BNT} and \cite{na}). M\'oricz and Siddiqi \cite{Mor} investigated the approximation properties of
some special N\"orlund means of Walsh-Fourier series of $L^{p}$ functions in
norm. In particular, they proved that if $f\in L^p(G),$ $1\leq p\leq \infty,$ $n=2^j+k,$ $1\leq k\leq 2^j \ (n\in \mathbb{N}_+)$ and $(q_k,k\in \mathbb{N})$ is a sequence of non-negative numbers, such that
\begin{equation}\label{MScond}
\frac{n^{\gamma-1}}{Q_n^{\gamma}}\sum_{k=0}^{n-1}q^{\gamma}_k =O(1),\ \ \text{for some} \ \ 1<\gamma\leq 2,
\end{equation}
then
\begin{equation} \label{MSes}
\Vert t_nf-f\Vert_p\leq \frac{C_p}{Q_n} \sum_{i=0}^{j-1}2^iq_{n-2^i}\omega_p\left(\frac{1}{2^i},f\right)+C_p\omega_p\left( \frac{1}{2^j},f\right),
\end{equation}
when $(q_k,k\in \mathbb{N})$ is non-decreasing, while 

$$
\Vert t_nf-f\Vert_p\leq \frac{C_p}{Q_n} \sum_{i=0}^{j-1}\left(Q_{n-2^i+1}-Q_{n-2^{i+1}+1} \right)\omega_p\left(\frac{1}{2^i},f\right) +C_p\omega_p\left( \frac{1}{2^j},f\right),
$$
when  $(q_k,k\in \mathbb{N})$ is non-increasing.

In this paper we improve and complement a result by M\'oricz and Siddiqi \cite{Mor}. In particular, we prove that their estimate of the Nörlund means with respect to the Walsh system holds also without their additional condition. Moreover, we prove a similar approximation result  in Lebesgue spaces for any $1\leq p<\infty$.

The paper is organized as follows: The main results are presented, proved and discussed
in Section 3. In particular, Theorems 1, 2 and 3 are parts of this new approach. In order not to disturb the presentations in Section 3, we use Section 2 for some necessary preliminaries.

\section{Preliminaries}

Let  $\mathbb{N}_{+}$ denote the set of the positive integers, $\mathbb{N}:=
\mathbb{N}_{+}\cup \{0\}.$ Denote by 
$
Z_{2}:=\{0,1\}
$
the additive group of integers modulo $2.$

Define the group $G$ as the complete direct product of the group $%
Z_{2}$ with the product of the discrete topologies of $Z_{2}\textrm{'s}$.
The direct product $\mu $ of the measures 
$
\mu^*(\left( \{j\}\right):=1/2\text{ \ }(j\in Z_{2})
$
is the Haar measure on $G$ with $\mu \left( G\right) =1.$

The elements of $G$ are represented by the sequences 
\begin{equation*}
x:=(x_{0},x_{1},\dots,x_{k},\dots)\qquad \left( \text{ }x_{k}\in
Z_{2}\right).
\end{equation*}
It is easy to give a base for the neighborhood of $G$, namely
\begin{eqnarray*}
I_{0}\left( x\right):=G, \ \ \
I_{n}(x):=\{y\in G\mid y_{0}=x_{0},\dots,y_{n-1}=x_{n-1}\}\text{ }(x\in
G,\text{ }n\in \mathbb{N}).
\end{eqnarray*}
Denote $I_n(0)$ by $I_n$ i.e $I_n :=I_n(0).$ It is well-known that every $n\in \mathbb{N}$ can be uniquely expressed as
$$
n=\sum_{k=0}^{\infty }n_{j}2^{j}, \ \ \ \text{where } \ \ \ n_{j}\in Z_{2}  \ \ \ (j\in \mathbb{
N})$$ 
and only a finite number of $n_{j}`$s differ from zero. 


First define the Rademacher functions as 

\begin{equation*}
r_{k}\left( x\right):={(-1)}^{x_k}, \text{ \ \ }\left(k\in \mathbb{N}\right).
\end{equation*}
Now we define the Walsh system $w:=(w_{n}:n\in \mathbb{N})$ on $G$ as

\begin{equation*}
w _{n}\left( x\right):=\prod_{k=0}^{\infty }r_{k}^{n_{k}}\left( x\right) 
\text{ \quad }\left( n\in \mathbb{N}\right).
\end{equation*}

The Walsh system is orthonormal and complete in $L^{2}\left( G\right)
\,$ (see e.g. \cite{sws}).

If $f\in L^{1}\left( G\right) $, then we can define the Fourier
coefficients, the partial sums of the Fourier series, the Fejér means, the Dirichlet and Fejér kernels 
with respect to the Walsh system in
the usual manner:
\begin{eqnarray*}
	\widehat{f}\left( k\right) &:&=\int_{G}fw_{k}d\mu ,\,%
	\text{\quad }\left(k\in \mathbb{N}\right) , \\
	S_{n}f &:&=\sum_{k=0}^{n-1}\widehat{f}\left( k\right) w _{k},\text{
		\quad }\left(n\in \mathbb{N}_{+},\text{ }S_{0}f:=0\right) , \\
	\sigma _{n}f &:&=\frac{1}{n}\sum_{k=1}^{n}S_{k}f,\text{ \quad \ \  }\left(n\in \mathbb{N}_{+}\right).\\
 D_{n}&:&=\sum_{k=0}^{n-1}\psi_{k},\text{ \quad \ \  }\left(n\in \mathbb{N}_{+}\right).\\
\end{eqnarray*}
\begin{eqnarray*}
 K_n&:&=\frac{1}{n}\sum_{k=1}^{n}D_{k},\text{ \quad \ \  }\left(n\in \mathbb{N}_{+}\right).\
\end{eqnarray*}
Recall that for Dirichlet and Fejér kernels $D_n$ and $K_n$ we have that (see e.g. \cite{gat})
\begin{equation}\label{dn2.3}
\quad \hspace*{0in}D_{2^{n}}\left( x\right) =\left\{
\begin{array}{l}
\text{ }2^{n},\text{\thinspace \thinspace \thinspace  if\thinspace
	\thinspace }x\in I_{n}, \\
\text{ }0,\text{\thinspace \thinspace \thinspace \thinspace \thinspace \thinspace \thinspace \thinspace if
	\thinspace \thinspace }x\notin I_{n},
\end{array}
\right.  
\end{equation}

\begin{equation}\label{dn2.4}
D_{2^n-m}\left( x \right)=D_{2^n}\left( x \right)-w_{2^n-1}\left( x \right){{D}_{m}}\left( x \right),\,\,\,0\le m<{2^{n}}
\end{equation}

\begin{equation} \label{fn5}
n\left\vert K_n\right\vert\leq
3\sum_{l=0}^{\vert n\vert } 2^l \left\vert K_{2^l} \right\vert,
\end{equation}
where $|n|=:\max\{j\in\mathbb{N}, n_j\neq0\}$ and

\begin{eqnarray} \label{fn40}
\int_{G} K_n (x)d\mu(x)=1,  \ \ \ \ \ 
\sup_{n\in\mathbb{N}}\int_{G}\left\vert K_n(x)\right\vert d\mu(x)\leq 2.
\end{eqnarray}
Moreover, if $n>t,$ $t,n\in \mathbb{N},$ then 

\begin{equation}\label{lemma2}
K_{2^n}\left(x\right)=\left\{ \begin{array}{ll}
2^{t-1},& x\in I_t\backslash I_{t+1},\quad x-e_t\in I_n, \\
\frac{2^n+1}{2}, & x\in I_n, \\
0, & \text{otherwise.} \end{array} \right.
\end{equation}

The $n$-th N\"orlund  mean $t_n$ of the Fourier series of $f$ is defined by
\begin{equation} \label{1.2}
t_nf:=\frac{1}{Q_n}\overset{n}{\underset{k=1}{\sum }}q_{n-k}S_kf, 
\ \ \ \text{where } \ \ \ 
Q_n:=\sum_{k=0}^{n-1}q_k.
\end{equation}
Here $\{q_k:k\geq 0\}$ is a sequence of nonnegative numbers, where $q_0>0$ and
$$
\lim_{n\rightarrow \infty }Q_{n}=\infty .
$$
Then the summability method (\ref{1.2}) generated by $\{q_k:k\geq 0\}$ is regular if and only if (see \cite{moo})
\begin{equation*}
\underset{n\rightarrow \infty }{\lim }\frac{q_{n-1}}{Q_{n}}=0.  \label{1a11}
\end{equation*}

The representation
\begin{equation*}
t_nf\left(x\right)=\underset{G}{\int}f\left(t\right)F_n\left(x-t\right) d\mu\left(t\right)
\end{equation*}
play central roles in the sequel, where 
\begin{equation}\label{1.3T}
F_n=:\frac{1}{Q_n}\overset{n}{\underset{k=1}{\sum }}q_{n-k}D_k
\end{equation}
is called the kernels of the N\"orlund  means.

It is well-known (see e.g. \cite{PTWbook}) that  every  N\"orlund summability method generated by non-increasing  sequence $(q_k,k\in \mathbb{N})$ is regular, but N\"orlund means generated by non-decreasing  sequence $(q_k,k\in \mathbb{N})$ is not always regular. In this paper we investigate regular N\"orlund  means only.

If we invoke Abel transformation we get the following identities: 
\begin{eqnarray}  \label{2b}
Q_n:=\overset{n-1}{\underset{j=0}{\sum}}q_j=\overset{n}{\underset{j=1}{%
		\sum }}q_{n-j}\cdot 1 =\overset{n-1}{\underset{j=1}{\sum}}%
\left(q_{n-j}-q_{n-j-1}\right) j+q_0n
\end{eqnarray}
and
\begin{equation}  \label{2bbb}
t_nf=\frac{1}{Q_n}\left(\overset{n-1}{\underset{j=1}{\sum}}\left(
q_{n-j}-q_{n-j-1}\right) j\sigma_{j}f+q_0n\sigma_nf\right).
\end{equation}
\section{Main Results}
Based on estimate \eqref{aaa} we can prove our first main results:
\begin{theorem}\label{Corollary3nnconv} 
	Let $2^N\leq n< 2^{N+1}$ and ${{t}_{n}}$ be a regular N\"orlund  mean generated by non-decreasing sequence $\{q_k:k\in \mathbb{N}\},$ in sign $q_k \uparrow.$ Then, for any  $f\in L^p(G),$ where $1\leq p< \infty, $ the following inequality holds:
\begin{equation*}
    \Vert t_nf-f\Vert_p\leq \frac{18}{Q_n} \sum_{i=0}^{N-1}2^iq_{n-2^i}\omega_p\left(\frac{1}{2^i},f\right)+12\omega_p\left( \frac{1}{2^N},f\right).
\end{equation*}
\end{theorem}
\begin{proof} 
	 Let $2^N\leq n<2^{N+1}.$ Since  $t_n$ are regular N\"orlund  means, generated by sequences of non-decreasing numbers $\{q_k:k\in \mathbb{N}\}$  by combining \eqref{2b} and \eqref{2bbb}, we can conclude that
	\begin{eqnarray*}
		&&\Vert t_nf(x)-f(x)\Vert_p \\
		&\leq&\frac{1}{Q_n}\left(\overset{n-1}{\underset{j=1}{\sum}}\left(q_{n-j}-q_{n-j-1}\right)j\Vert\sigma_jf(x)-f(x)\Vert_p+q_0n\Vert\sigma_nf(x)-f(x)
		\Vert_p\right)\\
		&:=&I+II,
	\end{eqnarray*}
Furthermore,
\begin{eqnarray*}
I&=&\frac{1}{Q_n}\overset{2^N-1}{\underset{j=1}{\sum}}\left(q_{n-j}-q_{n-j-1}\right)j\Vert\sigma_jf(x)-f(x)\Vert_p\\
&+&\frac{1}{Q_n}\overset{n-1}{\underset{j=2^N}{\sum}}\left(q_{n-j}-q_{n-j-1}\right)j\Vert\sigma_jf(x)-f(x)\Vert_p:=I_1+I_2.
\end{eqnarray*}

Now we estimate each terms separately. By using \eqref{aaa} for $I_1$ we obtain that\newline

\begin{eqnarray}\label{I_1}
I_1&\leq&\frac{3}{Q_n}\overset{N-1}{\underset{k=0}{\sum}}\overset{2^{k+1}-1}{\underset{j=2^k}{\sum}}\left(q_{n-j}-q_{n-j-1}\right)j  
\sum_{s=0}^{k} \frac{2^{s}}{2^{k}}\omega_p\left(1/2^s,f\right)\\ \nonumber
&\leq&\frac{3}{Q_n}\overset{N-1}{\underset{k=0}{\sum}}2^{k+1}\overset{2^{k+1}-1}{\underset{j=2^k}{\sum}}\left(q_{n-j}-q_{n-j-1}\right) 
\sum_{s=0}^{k} \frac{2^{s}}{2^{k}}\omega_p\left(1/2^s,f\right)\\\nonumber
&\leq&\frac{6}{Q_n}\overset{N-1}{\underset{k=0}{\sum}}\left(q_{n-2^k}-q_{n-2^{k+1}}\right) \sum_{s=0}^{k} 2^{s}\omega_p\left(1/2^s,f\right)\\\nonumber
&\leq&\frac{6}{Q_n}\overset{N-1}{\underset{s=0}{\sum}}2^{s}\omega_p\left(1/2^s,f\right)\sum_{k=s}^{N-1} \left(q_{n-2^k}-q_{n-2^{k+1}}\right)\\\nonumber
&\leq&\frac{6}{Q_n}\overset{N-1}{\underset{s=0}{\sum}}2^{s}q_{n-2^s}\omega_p\left(1/2^s,f\right).
\end{eqnarray}

It is evident that
\begin{eqnarray}\label{I_2}
	I_2&\leq&\frac{3}{Q_n}\overset{n-1}{\underset{j=2^N}{\sum}}\left(q_{n-j}-q_{n-j-1}\right)j  
	\sum_{s=0}^{N} \frac{2^{s}}{2^{N}}\omega_p\left(1/2^s,f\right)\\\nonumber
&\leq&\frac{3\cdot2^{N+1}}{Q_n}\overset{n-1}{\underset{j=2^N}{\sum}}\left(q_{n-j}-q_{n-j-1}\right)
\sum_{s=0}^{N} \frac{2^{s}}{2^{N}}\omega_p\left(1/2^s,f\right)\\\nonumber
&\leq&\frac{6q_{n-2^N}}{Q_n}
\sum_{s=0}^{N}2^{s}\omega_p\left(1/2^s,f\right)\leq\frac{6}{Q_n}\overset{N}{\underset{s=0}{\sum}}2^{s}q_{n-2^s}\omega_p\left(1/2^s,f\right)\\\nonumber
&\leq&\frac{6}{Q_n}\overset{N-1}{\underset{s=0}{\sum}}2^{s}q_{n-2^s}\omega_p\left(1/2^s,f\right)+6\omega_p\left(1/2^N,f\right).
\end{eqnarray}

For $II$ we have that
\begin{eqnarray*}
	II\leq\frac{3q_02^{N+1}}{Q_n}  
	\sum_{s=0}^{N} \frac{2^{s}}{2^{N}}\omega_p\left(1/2^s,f\right)\leq \frac{6}{Q_n}  
	\sum_{s=0}^{N-1}2^{s}q_{n-2^s}\omega_p\left(1/2^s,f\right)+6\omega_p\left(1/2^N,f\right).
\end{eqnarray*}

	The proof is complete.
\end{proof}

Our next main result reads:
\begin{theorem}\label{Corollary3nnconv0} 
	Let   ${{t}_{n}}$ be N\"orlund  mean generated by a non-increasing sequence $\{q_k:k\in \mathbb{N}\}$, in sign $q_k \downarrow$.
	Then, for any $f\in L^p(G),$ where $1\leq p< \infty, $ the inequality
	\begin{eqnarray*}
		\Vert t_{2^n}f-f\Vert_p\leq\sum_{s=0}^{n-1}\frac{2^s}{2^n}\omega _p\left(1/2^s,f\right)+3\overset{n-1}{\underset{s=0}{\sum}}\frac{n-s}{2^{n-s}}\frac{q_{2^{s}}}{q_{2^{n}}}\omega _p\left(1/2^s,f\right)+3\omega _p\left(1/2^n,f\right)
	\end{eqnarray*}
holds.
\end{theorem}
\begin{proof}
	By using \eqref{dn2.4} we find that
	\begin{eqnarray} \label{1.21}
	t_{2^n}f=D_{2^n}\ast f-\frac{1}{Q_{2^n}}\overset{2^n-1}{\underset{k=0}{\sum }}q_{k}\left(\left( w_{2^n-1}D_k\right)\ast f\right).
	\end{eqnarray}
	By using Abel transformation we get that
	\begin{eqnarray} 	\label{2cc}
	t_{2^n}f&=&D_{2^n}\ast f-\frac{1}{Q_{2^n}}\overset{2^n-2}{\underset{j=0}{\sum}}\left(q_j-q_{j+1}\right) j((w_{2^n-1}K_j)\ast f)\\ \notag
	&-& \frac{1}{Q_{2^n}}q_{2^n-1}(2^n-1)(w_{2^n-1}K_{2^n-1}\ast f)\\ \notag
	&=&D_{2^n}\ast f-\frac{1}{Q_{2^n}}\overset{2^n-2}{\underset{j=0}{\sum}}\left(q_j-q_{j+1}\right) j((w_{2^n-1}K_j)\ast f)\\ \notag
	&-& \frac{1}{Q_{2^n}}q_{2^n-1}2^n(w_{2^n-1}K_{2^n}\ast f)\\ \notag
	&+& \frac{q_{2^n-1}}{Q_{2^n}}(w_{2^n-1}D_{2^n}\ast f)
	\end{eqnarray}
	and
	\begin{eqnarray} 	\label{2c}
	&&t_{2^n}f(x)-f(x)
	=\int_G(f(x+t)-f(x))D_{2^n}(t)dt\\ \notag
	&-& \frac{1}{Q_{2^n}}\overset{2^n-2}{\underset{j=0}{\sum}}\left(q_j-q_{j+1}\right) j\int_G\left(f(x+t)-f(x)\right)w_{2^n-1}(t)K_j(t)dt\\  \notag
	&-& \frac{1}{Q_{2^n}}q_{2^n-1}2^n\int_G(f(x+t)-f(x))w_{2^n-1}(t)K_{2^n}(t)dt\\ \notag
	&+& \frac{q_{2^n-1}}{Q_{2^n}}\int_G(f(x+t)-f(x))w_{2^n-1}(t)D_{2^n}(t)dt\\ \notag
	&:=&I+II+III+IV.
	\end{eqnarray}
	By combining generalized Minkowski's inequality and  equality \eqref{dn2.3}  we find that
	$$\Vert I\Vert _p\leq \int_{I_n}\Vert f(x+t)-f(x))\Vert_p D_{2^n}(t)dt\leq \omega_p\left(1/2^n,f\right). $$
and
	$$\Vert IV\Vert _p\leq \int_{I_n}\Vert f(x+t)-f(x))\Vert_p D_{2^n}(t)dt\leq \omega_p\left(1/2^n,f\right) .$$

 Since
\begin{equation}\label{Q_n}
    2^nq_{2^n-1}\leq Q_{2^n}, \  n\in \mathbb{N} ,
 \end{equation}
If we combine \eqref{lemma2}, \eqref{Q_n} and generalized Minkowski's inequality, then  we get that
	\begin{eqnarray*}\label{fejaprox2}
		\Vert III\Vert_p
		&\leq&\int_{G}\left\Vert f\left( x+t\right) -f\left( x\right) \right\Vert_p
		K_{2^{n}}\left( t\right) d\mu (t)\\ \notag
		&=&\int_{I_{n}}\left\Vert f\left( x+t\right) -f\left( x\right) \right\Vert_p
		K_{2^{n}}\left( t\right) d\mu (t)\\
		&+&\sum_{s=0}^{n-1}\int_{I_{n}\left( e_{s}\right)}\left\Vert f\left( x+t\right) -f\left( x\right) \right\Vert_p K_{2^{n}}\left(t\right) d\mu (t)\\ \notag
		&\leq&\int_{I_{n}}\left\Vert f\left( x+t\right)-f\left( x\right) \right\Vert_p\frac{2^{n}+1}{2}d\mu(t)\\ \notag
		&+&\sum_{s=0}^{n-1}2^s\int_{I_n\left(e_s\right) }\left\Vert f\left(x+t\right)-f\left(x\right) \right\Vert_pd\mu(t)\\ \notag
		\\ \notag
		&\leq& \omega _{p}\left( 1/2^{n},f\right) \int_{I_{n}}\frac{2^{n}+1}{2}d\mu(t)
		+\sum_{s=0}^{n-1}2^{s}\int_{I_n\left(e_s\right)}\omega _{p}\left( 1/2^s,f\right)d\mu(t)\\ \notag
		&\leq&\omega _{p}\left( 1/2^{n},f\right)+\sum_{s=0}^{n-1}\frac{2^s}{2^n}\omega _p\left(1/2^s,f\right)\\
		&\leq& \sum_{s=0}^{n}\frac{2^s}{2^n}\omega _p\left(1/2^s,f\right).
	\end{eqnarray*}
From this estimate also it follows that	
\begin{eqnarray}\label{fejaprox2sssaaa}
\ \ \ \ \ 	2^n\int_{G}\left\Vert f\left( x+t\right) -f\left( x\right) \right\Vert_p K_{2^{n}}\left( t\right) d\mu (t)\leq \sum_{s=0}^{n}2^s\omega _p\left(1/2^s,f\right).
\end{eqnarray}
Let $2^k\leq j\leq 2^{k+1}-1.$ By applying  \eqref{fn5} and \eqref{fejaprox2sssaaa} we find that	
	\begin{eqnarray}\label{fejaprox22a}
	&&\left\Vert j\int_{G}\left\vert f\left( x+t\right) -f\left( x\right) \right\vert
	K_{j}\left( t\right) d\mu (t)\right\Vert_p\\ \notag
	&\leq& 3\sum_{s=0}^{k}2^s\int_{G}\left\Vert f\left( x+t\right) -f\left( x\right) \right\Vert_p
	K_{2^{s}}\left( t\right) d\mu (t)\\
	&\leq& 3\sum_{l=0}^{k}\sum_{s=0}^{l}2^s\omega _p\left(1/2^s,f\right).
	\end{eqnarray}
	Hence, by applying  \eqref{fn5} and \eqref{fejaprox22a} we find that	
	\begin{eqnarray*}
		\Vert II\Vert_p&\leq& \frac{1}{Q_{2^n}}\overset{2^n-1}{\underset{j=0}{\sum}}\left(q_j-q_{j+1}\right) j\int_G\Vert f(x+t)-f(x)\Vert_p \vert K_j(t)\vert dt\\  \notag
		&\leq&\frac{1}{Q_{2^n}}\overset{n-1}{\underset{k=0}{\sum}}\overset{2^{k+1}-1}{\underset{j=2^k}{\sum}}\left(q_j-q_{j+1}\right) j\int_G\Vert f(x+t)-f(x)\Vert_p \vert K_j(t)\vert dt\\  \notag
		&\leq & \frac{3}{Q_{2^n}}\overset{n-1}{\underset{k=0}{\sum}}\overset{2^{k+1}-1}{\underset{j=2^k}{\sum}}\left(q_j-q_{j+1}\right)\sum_{l=0}^{k}\sum_{s=0}^{l}2^s\omega _p\left(1/2^s,f\right)
		\\  \notag
		&\leq & \frac{3}{Q_{2^n}}\overset{n-1}{\underset{k=0}{\sum}}\left(q_{2^k}-q_{2^{k+1}}\right) \sum_{l=0}^{k}\sum_{s=0}^{l}2^s\omega _p\left(1/2^s,f\right)
		\\  \notag
		&\leq & \frac{3}{Q_{2^n}}\overset{n-1}{\underset{l=0}{\sum}}\sum_{k=l}^{n-1}\left(q_{2^k}-q_{2^{k+1}}\right)\sum_{s=0}^{l}2^s\omega _p\left(1/2^s,f\right)
		\\  \notag
		&\leq & \frac{3}{Q_{2^n}}\overset{n-1}{\underset{l=0}{\sum}}q_{2^l}\sum_{s=0}^{l}2^s\omega _p\left(1/2^s,f\right) 
		\leq  \frac{3}{Q_{2^n}}\overset{n-1}{\underset{s=0}{\sum}}2^s\omega _p\left(1/2^s,f\right)\sum_{l=s}^{n-1}q_{2^l}
		\\ \notag
		&\leq & \frac{3}{Q_{2^n}}\overset{n-1}{\underset{s=0}{\sum}}2^s\omega _p\left(1/2^s,f\right)q_{2^{s}}(n-s)\leq 3 \overset{n-1}{\underset{s=0}{\sum}}\frac{n-s}{2^{n-s}}\frac{q_{2^{s}}}{q_{2^{n}}}\omega _p\left(1/2^s,f\right).
	\end{eqnarray*}
	The proof is complete.	
\end{proof}

Finally, we state and prove the third main result.
\begin{theorem}\label{Corollary3nnconv1} 
Let $2^N\leq n< 2^{N+1}$ and  ${{t}_{n}}$ be N\"orlund  mean generated by non-increasing sequence $\{q_k:k\in \mathbb{N}\}$ (in sign $q_k \downarrow$) satisfying the condition
\begin{equation}\label{Cond}
\frac{1}{Q_n}=O\left(\frac{1}{n}\right),\text{\ \ as \ \ }n\rightarrow \infty
\end{equation}
	Then, for any $f\in L^p(G),$ where $1\leq p< \infty, $ we have the following inequality
\begin{eqnarray*}
\Vert t_nf-f\Vert_p \leq C\sum_{j=0}^{N}\frac{2^j}{2^N}\omega_p\left(1/2^j,f\right),
\end{eqnarray*}
where $C$ is a constant only depending on $p.$
\end{theorem}
\begin{proof} 
	Let $2^N\leq n<2^{N+1}.$ Since  $t_n$ is a regular N\"orlund  means, generated by a sequence of non-increasing numbers $\{q_k:k\in \mathbb{N}\}$  by combining \eqref{2b} and \eqref{2bbb}, we can conclude that	
	\begin{eqnarray*}
		&&\Vert t_nf(x)-f(x)\Vert_p \\
		&\leq&\frac{1}{Q_n}\left(\overset{n-1}{\underset{j=1}{\sum}}\left(q_{n-j-1}-q_{n-j}\right)j\Vert\sigma_jf(x)-f(x)\Vert_p+q_0n\Vert\sigma_nf(x)-f(x)
		\Vert_p\right)\\
		&:=&I+II.
	\end{eqnarray*}	
	Furthermore,
	\begin{eqnarray*}
		I&=&\frac{1}{Q_n}\overset{2^{N}-1}{\underset{j=1}{\sum}}\left(q_{n-j-1}-q_{n-j}\right)j\Vert\sigma_jf(x)-f(x)\Vert_p\\
		&+& \frac{1}{Q_n}\overset{n-1}{\underset{j=2^{N}}{\sum}}\left(q_{n-j-1}-q_{n-j}\right)j\Vert\sigma_jf(x)-f(x)\Vert_p
		:=I_1+I_2.
\end{eqnarray*}
Analogously to \eqref{I_1} we get that
\begin{eqnarray*}
	I_1&\leq&\frac{2}{Q_n}\overset{N-1}{\underset{k=0}{\sum}}\left(q_{n-2^{k+1}}-q_{n-2^{k}}\right)
	\sum_{s=0}^{k} 2^{s}\omega_p\left(1/2^s,f\right)\\
	&\leq&\frac{2}{Q_n}\sum_{s=0}^{N-1}2^{s}	\omega_p\left(1/2^s,f\right)\overset{N-1}{\underset{k=s}{\sum}}\left(q_{n-2^{k+1}}-q_{n-2^{k}}\right)\\
	&=&\frac{2}{Q_n}\sum_{s=0}^{N-1}2^{s}	\omega_p\left(1/2^s,f\right)(q_{n-2^{N}}-q_{n-2^s})\\
	&\leq&\frac{2q_{n-2^{N}}}{Q_n}\sum_{s=0}^{N-1}2^{s}	\omega_p\left(1/2^s,f\right)\leq\frac{2q_0}{Q_n}\sum_{s=0}^{N-1}2^{s}	\omega_p\left(1/2^s,f\right).
\end{eqnarray*}
Moreover, analogously to \eqref{I_2} we find that

\begin{eqnarray*}
I_2&\leq&\frac{2}{Q_n}\overset{n-1}{\underset{j=1}{\sum}}\left(q_{n-j-1}-q_{n-j}\right)j  \sum_{s=0}^{N} \frac{2^{s}}{2^{N}}\omega_p\left(1/2^s,f\right)\\
&=&\frac{2}{Q_n}\left(nq_0-Q_{n}\right)\sum_{s=0}^{N} \frac{2^{s}}{2^{N}}\omega_p\left(1/2^s,f\right)\\
&\leq &\frac{2nq_0}{Q_n2^N}\sum_{s=0}^{N} 2^{s}\omega_p\left(1/2^s,f\right)\\
&\leq&\frac{2q_0}{Q_n}\sum_{s=0}^{N} 2^{s}\omega_p\left(1/2^s,f\right).
\end{eqnarray*}
	
For $II$ we have that
	\begin{eqnarray*}
II&\leq&\frac{q_02^{N+1}}{Q_n}\sum_{s=0}^{N} \frac{2^{s}}{2^{N}}\omega_p\left(1/2^s,f\right)\\
&\leq &\frac{2q_0}{Q_n} \sum_{s=0}^{N}2^{s}\omega_p\left(1/2^s,f\right)
\end{eqnarray*}
Hence, by using \eqref{Cond} we obtain the required inequality above so the proof is complete.
\end{proof}
As a consequence we obtain the following similar result proved in M\'oricz and Siddiqi \cite{Mor}:
\begin{corollary}
Let $\{ q_k : k \geq 0\}$ be a sequence of non-negative numbers such
that in the case $q_k \uparrow$ condition 
	\begin{equation} \label{fn00}
\frac{q_{n-1}}{Q_n}=O\left(\frac{1}{n}\right),\text{\ \ as \ \ }n\rightarrow \infty.
\end{equation}
 is satisfied, while in case $q_k \downarrow$ condition \eqref{Cond} is satisfied. If $f\in Lip(\alpha, p)$ for some $\alpha > 0$ and $1 \leq p< \infty,$ then
\begin{equation}\label{dn2.30}
\Vert t_nf-f\Vert_p=\left\{
\begin{array}{l}
\text{ }O(n^{-\alpha}),\text{\qquad \qquad  if\qquad  }0<\alpha<1, \\
\text{ }O(n^{-1}\log n),\text{\qquad  if
	\quad \ }\alpha=1,\\
\text{ }O(n^{-1}),\text{\qquad \qquad   if \quad \ \
	}\alpha>1,
\end{array}
\right.  
\end{equation}
\end{corollary}
As a consequence we obtain the following similar result proved in M\'oricz and Siddiqi \cite{Mor}:
\begin{corollary} a) Let  ${{t}_{n}}$ be  N\"orlund  means generated by non-decreasing sequence $\{q_k:k\in \mathbb{N}\}$ satisfying regularity condition \eqref{fn00}. Then for some $f\in L^p(G),$ where $1 \leq p< \infty,$ 
	$$ \underset{n\rightarrow \infty }{\lim }\Vert {{t}_{n}}f-f\Vert_p\to 0, \ \ \ \text{as} \ \ \ n\to \infty.$$
	b)	Let  ${{t}_{n}}$ be N\"orlund  mean generated by non-increasing sequence $\{q_k:k\in \mathbb{N}\}$ satisfying condition \eqref{Cond}. Then for some $f\in L^p(G),$ where  $1 \leq p< \infty,$  
	$$ \underset{n\rightarrow \infty }{\lim }\Vert {{t}_{n}}f -f\Vert_p\to 0, \ \ \ \text{as} \ \ \ n\to \infty.$$
\end{corollary}

\end{document}